\documentclass[12pt, a4paper]{amsart}

\usepackage[english]{babel}
\usepackage{amsmath}
\usepackage{amssymb}
\usepackage{amscd} 
\usepackage{hyperref} 
\usepackage{microtype} 
\usepackage{verbatim} 
\usepackage{color}  
\usepackage{tikz-cd}\usetikzlibrary{babel} 
\usepackage{enumerate} 
\usepackage{mathtools} 
\usepackage{cite}
\usepackage[initials,msc-links]{amsrefs}
\usepackage[all]{xy}

\title{Adelic superrigidity and profinitely solitary lattices}

\author[H. Kammeyer]{Holger Kammeyer}
\author[S. Kionke]{Steffen Kionke}
 
 \address[H. Kammeyer]{Institute for Algebra and Geometry, Karlsruhe Institute of Technology, 76131 Karlsruhe, Germany}
 \email{holger.kammeyer@kit.edu}

\address[S. Kionke]{FernUniversit\"at in Hagen, Fakult\"at Mathematik und Informatik,
Lehrgebiet Algebra,
58084 Hagen, Germany}
\email{steffen.kionke@fernuni-hagen.de}
 

\theoremstyle{plain}
\newtheorem{theorem}{Theorem}
\newtheorem{lemma}[theorem]{Lemma}
\newtheorem{corollary}[theorem]{Corollary}

\theoremstyle{definition}
\newtheorem{definition}[theorem]{Definition}
\newtheorem{remark}[theorem]{Remark}
\newtheorem{example}[theorem]{Example}

\numberwithin{equation}{section}
\numberwithin{theorem}{section}

\providecommand{\ignore}[1]{}
\providecommand{\alg}[1]{\mathbf{#1}}

\providecommand{\R}{\mathbb{R}}
\providecommand{\Q}{\mathbb{Q}}

\providecommand{\C}{\mathbb{C}}
\providecommand{\K}{\mathbb{K}}

\newcommand{\RN}[1]{\textup{\uppercase\expandafter{\romannumeral#1}} 
}

\newcommand*{\arXiv}[1]{ \href{http://www.arxiv.org/abs/#1}{arXiv:\textbf{#1}}}

\DeclareMathOperator{\id}{Id}

\DeclareMathOperator{\Hom}{Hom}

\DeclareMathOperator{\Aut}{Aut}

\DeclareMathOperator{\pr}{pr}

\providecommand{\normal}{\trianglelefteq}

\providecommand{\fg}{\mathfrak{g}}
\providecommand{\fh}{\mathfrak{h}}

\providecommand{\bbN}{\mathbb{N}}
\providecommand{\bbR}{\mathbb{R}}
\providecommand{\bbQ}{\mathbb{Q}}
\providecommand{\bbZ}{\mathbb{Z}}

\providecommand{\bbA}{\mathbb{A}}
\providecommand{\bbC}{\mathbb{C}}
\DeclareMathOperator{\SL}{SL}

\providecommand{\bHom}{\underline{\Hom}_{\text{gr}}}

\begin{document}
\selectlanguage{english}

\begin{abstract}
By arithmeticity and superrigidity, a commensurability class of lattices in a higher rank Lie group is defined by a unique algebraic group over a unique number subfield of \(\R\) or \(\C\).  We prove an adelic version of superrigidity which implies that two such commensurability classes define the same profinite commensurability class if and only if the algebraic groups are adelically isomorphic.  We discuss noteworthy consequences on profinite rigidity questions.
\end{abstract}

\maketitle

\section{Introduction}

Let \(\K\) be either \(\R\) or \(\C\) and let \(\mathbf{G}\) be a connected absolutely almost simple linear algebraic \(\K\)-group with \(\mathrm{rank}_\K \mathbf{G} \ge 2\).  

\begin{definition} \label{def:g-arithmetic-pair}
A \emph{\(\mathbf{G}\)-arithmetic pair} \((K, \mathbf{H})\) consists of a dense subfield \(K \subset \mathbb{K}\) with \([K : \Q] < \infty\) and a connected simply connected absolutely almost simple linear algebraic \(K\)-group \(\mathbf{H}\) which is \(\K\)-isogenous to \(\mathbf{G}\) and anisotropic at all other infinite places of \(K\).
\end{definition}

For any \(\mathbf{G}\)-arithmetic pair \((K, \mathbf{H})\) the arithmetic subgroups of \(\mathbf{H}(K)\) 
determine a unique commensurability class of lattices \(\Gamma_{(K, \mathbf{H})} \subset \mathbf{G}(\mathbb{K})\). A reformulation of Margulis seminal work provides a commensurability classification of lattices in \(\mathbf{G}(\mathbb{K})\), where \emph{arithmeticity} gives \emph{surjectivity}, and \emph{superrigidity} gives \emph{injectivity} in the following statement.
\begin{theorem}[Margulis' commensurability classification] \label{thm:commensurability-classification}
  Assigning \((K, \mathbf{H}) \mapsto \Gamma_{(K, \mathbf{H})}\) defines a bijection \(\Phi\) from isomorphism classes of \(\mathbf{G}\)-arithmetic pairs to commensurability classes of lattices in \(\mathbf{G}(\K)\).
\end{theorem}
Here, two \(\mathbf{G}\)-arithmetic pairs are \emph{isomorphic} if the subfields are equal and the groups are isomorphic, cf.\,p.\,\pageref{page:g-arithmetic-pairs-iso}.  In~\cite{Kammeyer-Kionke:rigidity}, we recently showed that for most (but not all) groups \(\mathbf{G}\), \emph{profinite commensurability} is a strictly weaker equivalence relation on the set of lattices than commensurability.  The main result of this article is a profinite commensurability classification of lattices in terms of \(\mathbf{G}\)-arithmetic pairs; provided that \(\mathbf{G}\) has the \emph{congruence subgroup property} (CSP).

\begin{theorem}[Profinite commensurability classification] \label{thm:profinite-commensurability-classification}
  If \(\mathbf{G}\) has CSP, then \(\Phi\) descends to a bijection from \emph{local isomorphism classes} of \(\mathbf{G}\)-arithmetic pairs to \emph{profinite commensurability classes} of lattices.
\end{theorem}

We say that \(\mathbf{G}\) has \emph{CSP} if the \(K\)-group \(\mathbf{H}\) in \emph{every} \(\mathbf{G}\)-arithmetic pair has finite congruence kernel.  According to a conjecture of Serre~\cite{Platonov-Rapinchuk:algebraic-groups}*{(9.45), p.\,556}, requiring CSP is redundant under our assumption that \(\mathrm{rank}_{\mathbb{K}} \mathbf{G} \ge 2\).  The conjecture remains open only for certain anisotropic forms of type \(A_n\) and \(E_6\).  A recent survey with the precise definitions and statements can be found in \cite{Kammeyer-Kionke:rigidity}*{Appendix~A}.  Two \mbox{\(\mathbf{G}\)-arithmetic} pairs are \emph{locally isomorphic} if the groups are isomorphic over a \emph{local isomorphism} of the number fields of definition by which we mean an isomorphism of the finite adele rings; cf.\,Definition~\ref{def:locally-isomorphic}.  Two groups \(\Gamma\) and \(\Lambda\) are \emph{profinitely commensurable} if an open subgroup of \(\widehat{\Gamma}\) is isomorphic to an open subgroup of \(\widehat{\Lambda}\).

The two theorems reveal that the arithmetic of the number field \(K\) in a \(\mathbf{G}\)-arithmetic pair \((K, \mathbf{H})\), as well as the realization as subfield of \(\mathbb{K}\), play a prominent role in the classification.  This becomes particularly apparent for real forms \(\mathbf{G}\) of the exceptional types \(E_8\), \(F_4\), and \(G_2\) because these Cartan--Killing types have unique \(\mathfrak{p}\)-adic forms.  For such~\(\mathbf{G}\), we will see that consequently, any given totally real number subfield \(K \subset \mathbb{R}\) sits in a unique \(\mathbf{G}\)-arithmetic pair \((K, \mathbf{H})\).  So we may write \(\Gamma_{(K, \mathbf{H})} = \Gamma_K\) and the profinite commensurability classification of lattices translates entirely into a problem of algebraic number theory.

\begin{theorem} \label{thm:exceptional-cases}
  Let \(\mathbf{G}\) be the simply connected \(\R\)-group of type \(F_{4(4)}\), \(G_{2(2)}\), \(E_{8(8)}\), or \(E_{8(\text{-}24)}\).  Then assigning \(K \mapsto \Gamma_K\) defines a bijection
  \begin{enumerate}[(i)]
  \item \label{item:comm-subfields} from totally real number subfields of~\(\R\) to commensurability classes of lattices in \(\mathbf{G}(\R)\), which descends to a bijection
  \item \label{item:prof-comm-adelic} from local isomorphism classes of totally real number fields to profinite commensurability classes of lattices in \(\mathbf{G}(\R)\).
  \end{enumerate}  
\end{theorem}

Hence in the situation of the theorem, every realization of \(K\) as subfield of \(\R\) defines a distinct commensurability class of lattices in \(\mathbf{G}(\R)\).  But all these lattices are profinitely commensurable.  Additional distinct commensurability classes in the same profinite commensurability class arise if \(K\) has a non-isomorphic but locally-isomorphic sibling.  Recall, however, that \(K\) is Galois over \(\Q\) if and only if every embedding of \(K\) into \(\R\) has the same image.  Moreover, Galois extensions have no locally isomorphic siblings~\cite{Klingen:similarities}*{Theorem~III.1.4.(b)}.  We conclude a neat description of the \emph{profinitely solitary} lattices in \(\mathbf{G}(\R)\) all of whose profinitely commensurable lattices in \(\mathbf{G}(\R)\) are actually commensurable.

\begin{corollary} \label{cor:solitary-galois}
  Let \(\mathbf{G}\) be as in Theorem~\ref{thm:exceptional-cases} and \(K \subset \R\) be totally real.  Then \(\Gamma_K \subset \mathbf{G}(\R)\) is profinitely solitary if and only if \(K / \Q\) is Galois.
\end{corollary}

Note however that the profinite solitude of \(\Gamma_K \subset \mathbf{G}(\mathbb{K})\) for \(K/\Q\) Galois does not extend beyond lattices in the fixed group \(\mathbf{G}(\R)\).  In fact, our proof shows that for every totally real \(K \subset \R\), Galois or not, the corresponding lattices in \(E_{8(8)}\) and \(E_{8(-24)}\) are profinitely commensurable but they are not abstractly commensurable by superrigidity.

Theorem~\ref{thm:exceptional-cases} as well as Corollary~\ref{cor:solitary-galois} would also hold true for the rank one group \(F_{4(-20)}\) if it was known that it had CSP.  This is of interest in view of a question asked by A.\,Reid in~\cite{Reid:profinite}*{Question~10}:

\smallskip
\emph{``Are lattices in rank one semisimple Lie groups profinitely rigid?''}
\smallskip

In other words, does such a lattice \(\Gamma\) have the property that for every finitely generated residually finite group \(\Lambda\), the assumption \(\widehat{\Gamma} \cong \widehat{\Lambda}\) implies \(\Gamma \cong \Lambda\)? In~\cite{Stover:lattices-pu}, M.\,Stover answered this question in the negative by constructing counterexamples in the Lie groups \(\mathrm{PU}(n,1)\) for all \(n \ge 2\).   These Lie groups do not have CSP.  But Reid already speculated himself that the answer should be ``no'' for rank one groups with CSP, too, and indeed, if \(F_{4(-20)}\) has CSP, then Corollary~\ref{cor:solitary-galois} shows that the answer is a resounding no: every totally real non-Galois extension of \(\Q\) would yield a counterexample.  Similarly, the construction in \cite{Kammeyer-Kionke:rigidity}*{Section~3.5} would give counterexamples in the rank one Lie group \(\mathrm{Sp}(n,1)\) if it has CSP.  Contrary to Serre's conjecture (uttered in the 70s), it has meanwhile been suspected that \(\mathrm{Sp}(n,1)\) and \(F_{4(-20)}\) might have CSP because their lattices exhibit other higher rank properties: arithmeticity, superrigidity, and property~(T)~\cite{Lubotzky:non-arithmetic-rigid}*{Section~4}.

For other types of groups \(\mathbf{G}\), lattices will generally not be profinitely solitary even if their number field is a Galois extension of \(\Q\).  However, for $K = \bbQ$ we have the following result.

\begin{theorem} \label{thm:rational-field-solitary}
Suppose \(\mathbb{K} = \R\) and \(\mathbf{G}\) is not of type \(A_n\), \(D_n\), or \(E_6\).  Then for every \(\mathbf{G}\)-arithmetic pair \((\Q, \mathbf{H})\), the lattice \(\Gamma_{(\Q, \mathbf{H})} \subset \mathbf{G}(\R)\) is profinitely solitary. In particular, every non-cocompact lattice in \(\mathbf{G}(\R)\) is profinitely solitary.
\end{theorem}

For groups of type \(A_n\), \(D_n\), and \(E_6\) the situation is subtle and seems to require a more thorough inspection. For instance, profinitely isomorphic, non-commensurable, and non-cocompact lattices in $\SL_n(\bbR)$, $\SL_n(\bbC)$ and $\SL_n(\mathbb{H})$ where $n \geq 6$ is a composite number were constructed in \cite[Proposition 3.1]{Kammeyer-Kionke:rigidity}. 

We saw in \cite{Kammeyer-Kionke:rigidity} that for \(\mathbb{K} = \C\), all lattices in \(\mathbf{G}(\C)\), cocompact or not, are profinitely solitary if the complex group \(\mathbf{G}\) has type \(E_8\), \(F_4\), or \(G_2\).  In contrast, if the complex group \(\mathbf{G}\) has any other type except possibly \(E_6\) and \(A_n\) where CSP is open, then a non-cocompact lattice \(\Gamma \subset \mathbf{G}(\C)\) is not necessarily profinitely solitary as we shall explain in Remark~\ref{remark:complex-case}.

\smallskip

The main tool underlying our results is an \emph{adelic superrigidity theorem}; see Theorem \ref{thm:adelic-superrigidity}.  The result relies only on Margulis superrigidity and has a noteworthy consequence indicating a relation to the congruence subgroup property; see Theorem \ref{thm:profinite-isomorphism-adelic}. 
\begin{corollary}
If two superrigid lattices have isomorphic profinite completions, then there are subgroups of finite index with isomorphic pro-congruence completions. 
\end{corollary}

\subsection*{Structure of the article.}  In Section~\ref{sec:tools}, we provide preliminary material from algebraic geometry.  Section~\ref{sec:adelic-superrigidity} is devoted to the statement and proof of the adelic superrigidty theorem.  In Section~\ref{sec:classification}, we deduce the profinite commensurability classification of lattices and Section~\ref{sec:profinite-rigidity} concludes with the applications on profinite rigidity as outlined in this introduction.

\subsection*{Acknowledgements}  We gratefully acknowledge financial support by the German Research Foundation (DFG) via the Priority Program ``Geometry at Infinity'', DFG 338540207 and DFG 441848266, and via the Reseach Training Group ``Asymptotic Invariants and Limits of Groups and Spaces'', DFG 281869850.

\section{Tools from algebraic geometry} \label{sec:tools}
 
Our first goal is to establish the adelic superrigidity theorem (Theorem~\ref{thm:adelic-superrigidity}) which forms the technical core result of our investigation.  To this end we shall need some tools from algebraic geometry which we supply in this section.  The following convenient lemma is proven in~\cite{Borel-Tits:homomorphismes}*{1.4}.

\begin{lemma} \label{lemma:defined-over-k}
  Let \(K / k\) be a field extension with \(K\) algebraically closed.  Let \(f \colon X \rightarrow Y\) be a \(K\)-morphism of \(k\)-varieties such that \(f(D) \subset Y(k)\) for some Zariski dense subset \(D \subset X(k)\).  Then \(f\) is defined over \(k\).
\end{lemma}

Here is a refined version for \(k\)-algebras instead of fields.

\begin{lemma} \label{lemma:defined-over-subalgebra}
  Let $k$ be a field, let $E\subseteq B$ be $k$-algebras. Let $\mathfrak{X}$ be an affine scheme over $k$ and let $\mathfrak{Y}$ be an affine scheme over $E$.
  Let $f\colon \mathfrak{X} \times_k B \to \mathfrak{Y} \times_E B$ be a morphism of schemes over $B$.
  Assume there is a Zariski dense subset $D \subseteq \mathfrak{X}(k)$
  such that $f(D) \subseteq \mathfrak{Y}(E)$. Then $f$ is defined over $E$.
\end{lemma}
\begin{proof}
  Let $k[\mathfrak{X}]$ denote the coordinate ring of $\mathfrak{X}$. Similarly, we write
  $B[\mathfrak{X}] = B \otimes_k k[\mathfrak{X}]$ and
  $E[\mathfrak{X}] = E \otimes_k k[\mathfrak{X}]$. We note that the canonical map
  $E[\mathfrak{X}] \to B[\mathfrak{X}]$ is injective.

  Let $\gamma \in B[\mathfrak{X}]$. We
  claim that if $\gamma(\delta) \in E$ for all $\delta \in D$, then
  $\gamma \in E[\mathfrak{X}]$. To see this we take a $k$-basis $\{e_i\}_{i\in I}$ of $E$ and a basis $\{b_j\}_{j \in J}$ of some complement of $E$ in $B$. We may write
  \[ \gamma = \sum_{i\in I} e_i\otimes t_i + \sum_{j\in J} b_j \otimes s_j\]
  for certain (uniquely determined) elements $s_j, t_i \in k[\mathfrak{X}]$.
  If we apply this element $\gamma$ of the coordinate ring to $\delta \in D$, we obtain
  \[ E \ni \gamma(\delta) =  \sum_{i\in I} e_it_i(\delta) + \sum_{j\in J} b_j s_j(\delta).\]
  Since $t_i(\delta)$ and $s_j(\delta)$ lie in $k$, we deduce that the second term vanishes, i.e., $s_j(\delta) = 0$ for all $\delta \in D$ and all $j \in J$.
  Since $D \subseteq \mathfrak{X}$ is Zariski dense, we deduce that $s_j = 0$ for all $j \in J$.

  Finally, we can prove the lemma. Let $f^* \colon B[\mathfrak{Y}] \to B[\mathfrak{X}]$ be the dual morphism. We have to show that $f^*(E[\mathfrak{Y}]) \subseteq E[\mathfrak{X}]$. Let $t\in E[\mathfrak{Y}]$. By assumption we have
  $f^*(t)(\delta) = t(f(\delta)) \in E$ for all $\delta \in D$, hence the assertion follows from the observation above.
\end{proof}

\begin{lemma} \label{lemma:zariski-dense}
  Let $k$ be a local field with $\mathrm{char}(k) = 0$ and let $\mathfrak{X}$ be a smooth irreducible affine variety over $k$.
  \begin{enumerate}
  \item\label{it:open-dense} If the closure of $U \subseteq \mathfrak{X}(k)$has non-empty interior in the Euclidean topology, then
  $U$ is Zariski dense.
  \item\label{it:dimension-not-open} If $L_1,\dots,L_r$ are finite extension fields of $k$ with $\sum_{i=1}^r[L_i:k]\geq 2$, then the image of the diagonal map
  \[
  	\mathfrak{X}(k) \to \mathfrak{X}(L_1) \times \cdots \times \mathfrak{X}(L_r)
  \]
  is closed with empty interior.
  \end{enumerate}
\end{lemma}
\begin{proof}
\eqref{it:open-dense}:
	Since the Zariski topology is noetherian, there is a minimal Zariski closed subset $Z \subseteq \mathfrak{X}(\overline{k})$ such that
	$Z \cap \mathfrak{X}(k)$ contains a non-empty Hausdorff open set. We claim that $Z$ is $\mathfrak{X}$.
	Finite unions of nowhere dense sets are nowhere dense, thus $Z$ is irreducible by minimality. The singular locus of $Z$ has properly smaller dimension and hence
	cannot contain a euclidean open set by minimality of $Z$. On the complement, however, $Z(k)$ locally is a submanifold and can only contain
	a Hausdorff open set, if $\dim_k(Z) = \dim_k(\mathfrak{X})$. 
	  
  \eqref{it:dimension-not-open}:
  Since $\mathfrak{X}$ is smooth, $\mathfrak{X}(k)$ and  $\mathfrak{X}(L_1) \times \cdots \times \mathfrak{X}(L_r)$ are $k$-analytic manifolds of dimension $\dim_k(\mathfrak{X})$ and $\dim_k(\mathfrak{X})\sum_{i=1}^r[L_i:k]$ respectively\footnote{One way to see this is to use ``equality'' of analytic and algebraic tangent spaces.}. 
  The diagonal map is a closed immersion and so,  by assumption, the image is a submanifold of properly smaller dimension; in particular it has empty interior.\end{proof}

Let $\alg{G}$ and $\alg{H}$ be affine group schemes over some field $k$.
Define $\bHom(\alg{G},\alg{H})$ to be the functor which assigns to every commutative $k$-algebra $E$
the set of homomorphisms of group schemes $\alg{G}\times_k E \to \alg{H}\times_k E$. In general this functor need not be representable, but in some situations it is.
\begin{lemma} \label{lemma:lie-algebra-group-transfer}
  Let $k$ be a field of characteristic zero. Let $\alg{G}$ and $\alg{H}$ be linear algebraic groups defined over $k$ and assume that $\alg{G}$ is semisimple. Then the following assertions hold.
  \begin{enumerate}
  \item\label{it:hom-is-variety} The functor $\bHom(\alg{G},\alg{H})$ is an affine scheme defined over $k$. If $\alg{G}$ is simply connected, then
    taking the tangent map yields an isomorphism
    $\bHom(\alg{G},\alg{H})(k) \cong \Hom_{\text{Lie}}(\fg,\fh)$,
    where $\fg$ and $\fh$ denote the $k$-Lie algebras of  $\alg{G}$ and $\alg{H}$ respectively.

  \item\label{it:ev-algebraic} The evaluation map $\mathrm{ev} \colon \bHom(\alg{G},\alg{H}) \times \alg{G} \to \alg{H}$
    is algebraic.
    \end{enumerate}
    \end{lemma}
    \begin{proof}
    Assertion \eqref{it:hom-is-variety} is a direct consequence of \cite{SGA3-3}*{Proposition 7.3.1,
  Expos\'e XXIV}.  The evaluation $\mathrm{ev}$ is a natural transformation of
  functors, hence by Yoneda's Lemma is algebraic. 
    \end{proof}
   
\begin{lemma} \label{lemma:tuple-yields-morphism}
Let $K$, $L$ be two algebraic number fields.
Let $\alg{G}$ and $\alg{H}$ be linear algebraic groups defined over $K$ and $L$ respectively. Assume further that $\alg{G}$ is connected and semisimple.

Let $S$ be some (possibly infinite) set. Suppose that for each $v \in S$ there is a field $F_v$ with fixed inclusions
$i_v \colon K \to F_v$ and
$j_v \colon L \to F_v$ and a homomorphism $\alpha_v \colon \alg{G} \times_K F_v \to \alg{H}\times_L {F_v}$ of algebraic groups defined over
    $F_v$. 
    
    Then there is a unique morphism $f \colon \alg{G} \times_K B \to \alg{H} \times_L B$ of group schemes defined over $B = \prod_{v \in S} F_v$ such that the following diagram commutes
    \begin{equation*}
      \xymatrix{
        \alg{G}(B) \ar[d]_{\alg{G}(\pr_v)}\ar[r]^{f} & \alg{H}(B)\ar[d]^{\alg{H}(\pr_v)} \\
        \alg{G}(F_v) \ar[r]^{\alpha_v} & \alg{H}(F_v),
        }
      \end{equation*}
      for all $v \in S$ where $\pr_v\colon B \to F_v$ denotes the projection onto the $v$-component.
\end{lemma}

Let us note that the assertion in the lemma is obvious if $S$ is a finite set. The algebraic problem that has to be overcome in the general case is that tensor products and infinite direct products do not commute.
\begin{proof}
   Consider the finite dimensional separable $\bbQ$-algebra $A = K \otimes_\bbQ L$.
   We note that $A$ is a finite direct product of fields each containing a copy of $K$ and $L$.
   By Lemma \ref{lemma:lie-algebra-group-transfer}, which obviously extends to finite dimensional separable algebras, $\mathfrak{X} = \bHom(\alg{G} \times_K A,\alg{H}\times_L A)$ is an affine scheme over $A$.  Consider the
  evaluation map on the coordinate rings
  $\mathrm{ev}^* \colon A[\alg{H}] \to A[\mathfrak{X}]\otimes_A A[\alg{G}]$.
For every $v \in S$, $\alpha_v$ corresponds to an $F_v$-rational point in $\mathfrak{X}(F_v)$. This means that there is a 
 homomorphism $\alpha_v' \colon A[\mathfrak{X}] \to F_v$ of $k$-algebras such that
$\alpha_v^* = (\alpha_v', \id) \circ \mathrm{ev}^*$.
Taking the product $\prod_{v \in S} \alpha_v' \colon A[\mathfrak{X}] \to B$ we obtain a $B$-rational point of $\mathfrak{X}$, i.e., a homomorphism $f \colon \alg{G} \times_K B \to \alg{H}\times_L B$ of group schemes defined over $B$ such that $f^* = (\prod \alpha_v' \otimes \id) \circ \mathrm{ev}_B^*$. The above diagramm commutes since
\[ (\pr_v \otimes \id) \circ f^* \circ \iota = (\alpha_v' \otimes \id) \circ \mathrm{ev}^* = \alpha_v^*. \]
Here $\iota\colon L[\alg{H}] \to B[\alg{H}]$ denotes the inclusion.

Uniqueness: Suppose that $h \colon \alg{G} \times_K B \to \alg{H}\times_L B$ is another homomorphism such that the diagramm commutes, then
\[
	\pr_v \circ t \circ h^* =  \pr_v \circ t \circ f^*
\]
holds for all $B$-algebra homomorphisms $t \colon B[\alg{G}] \to B$. Since a homomorphism into $B$ is determined by all its components, 
we obtain $t \circ h^* = t \circ f^*$ for all $t \colon B[\alg{G}] \to B$. This implies $h^* = f^*$ using the following observation:\\[1ex]
\textit{$B$-rational points of $\alg{G}$ are Zariski dense in $\alg{G}\times_K B$, i.e.,} 
\[ \bigcap_{t \in \alg{G}(B)} \ker(t) = \{0\} \subseteq B[\alg{G}].\]
Let $x \in B[\alg{G}]$ with $x \neq 0$. Then $i(x) \neq 0$ where $i\colon B[\alg{G}] \to \prod_v F_v[\alg{G}]$ is the canonical injective homomorphism. 
Since some component of $i(x)$ is non-zero, the assertion follows from the well-known fact that
 $\alg{G}(F_v)$ is Zariski dense in $\alg{G} \times_K F_v$ for every $v \in S$; see \cite[18.3]{Borel:alg-groups}.
\end{proof}

\section{Adelic superrigidity} \label{sec:adelic-superrigidity}

In this section, let \(K\) be a number field and let \(\alg{G}\) be a connected simply connected absolutely almost simple \(K\)-group.  By the seminal work of Margulis~\cite{Margulis:discrete-subgroups}*{Theorem~(C), p.\,259}, \(\alg{G}\) is algebraically superrigid in the following sense provided that \(\sum_{v \in V_\infty(K)} \mathrm{rank}_{K_v} \alg{G} \ge 2\).

\begin{definition}\label{def:superrigidity}
  We call \(\alg{G}\) \emph{algebraically superrigid} if for every field \(l\) with \(\operatorname{char}(l) = 0\), every connected absolutely almost simple \(l\)-group \(\alg{H}\) and every homomorphism \(\delta \colon \Gamma \rightarrow \alg{H}(l)\) from an arithmetic subgroup \(\Gamma\) of \(\alg{G}\) with Zariski dense image, there exists a unique field embedding \(\sigma \colon K \rightarrow l\), a unique surjective \(l\)-morphism \(\eta \colon \alg{G}\times_K l\rightarrow \alg{H}\), and a unique group homomorphism \(\nu \colon \Gamma \rightarrow Z(\alg{H})(l)\) such that \(\delta(\gamma) = \nu(\gamma) \cdot \eta(\gamma)\).
\end{definition}

We remark that it is also well-known that CSP implies certain versions of superrigidity~\citelist{\cite{Bass-Milnor-Serre:solution}*{Section~16} \cite{Raghunathan:on-csp}*{Section~7}}.  For the groups \(\mathrm{Sp}(n,1)\) and \(F_{4(-20)}\), CSP is open as we remarked earlier, but superrigidity theorems were proven by Corlette~\cite{Corlette:superrigidity} and Gromov--Schoen~\cite{Gromov-Schoen:superrigidity}.  In each case, it is not hard to see that algebraic superrigidity in the sense of our definition follows so that \(\mathrm{Sp}(n,1)\) and \(F_{4(-20)}\) can be included into our discussion even though they have rank one.

Algebraic superrigidity asserts that a homomorphism from an arithmetic group to the \(l\)-points of an algebraic \(l\)-group extends to a morphism of the algebraic groups defined over \(l\).  The purpose of this section is to complement this theorem  with a local version: a homomorphims from an arithmetic group to the finite adele points of an algebraic group extends to a morphism of the algebraic group schemes defined over the finite adele rings.  We will use standard terminology to formulate the precise statement: \(V(K) = V_\infty(K) \cup V_f(K)\) denotes the set of places of \(K\) which is the disjoint union of infinite and finite places.  For any \(v \in V(K)\), the local completion of \(K\) at \(v\) is denoted \(K_v\) with valuation ring \(\mathcal{O}_v \subset K_v\).  The ring of finite adeles \(\mathbb{A}^f_K\) of \(K\) is by definition the locally compact subring of \(\prod_{v \in V_f(K)} K_v\) consisting of those elements with almost all coordinates in \(\mathcal{O}_v\).

\begin{theorem}[Adelic superrigidity]\label{thm:adelic-superrigidity}
Let $K$ be an algebraic number field, let \(\alg{G}\) be an absolutely almost simple algebraically superrigid $K$-group and let $\Gamma \subseteq \alg{G}$ be an arithmetic subgroup.

Suppose that a number field $L$,  a connected absolutely almost simple \(L\)-group \(\alg{H}\) and a homomorphism
$\varphi \colon \Gamma \to \alg{H}(\bbA^f_{L})$ are given such that $\overline{\varphi(\Gamma)}$ has non-empty interior.
Then there are
\begin{itemize}
\item an injective map $w \colon V_f(L) \to V_f(K)$,
\item isomorphisms of topological fields $j_v\colon L_v \to K_{w(v)}$ for all $v \in V_f(L)$ which induce an injective homomorphism $j \colon \bbA^f_L \to \bbA^f_K$ of topological rings,
\item a homomorphism of group schemes over $j$
\[\eta\colon \alg{G} \times_K {\bbA^f_K} \to \alg{H} \times_L \bbA_L^f, \]
\item a group homomorphism $\nu \colon \Gamma \to Z(\alg{H})(\bbA^f_L)$ with finite image
\end{itemize}
such that $\varphi(\gamma) = \nu(\gamma)\eta(\gamma)$ for all $\gamma \in \Gamma$.
In addition, $w$, $\eta$, $\nu$  and the isomorphisms $j_v$ for $v\in V_f(L)$ are uniquely determined by this condition.

Moreover, $[L:\bbQ] \leq [K:\bbQ]$ and if equality occurs, then $w$ is a bijection and $j$ is an isomorphism
\end{theorem}

As preparation for the proof, we discuss a local analogue.
\begin{lemma}[Local superrigidity]\label{lem:local-rigidity} 
Let $K$ be an algebraic number field, let \(\alg{G}\) be an absolutely almost simple algebraically superrigid $K$-group and let $\Gamma \subseteq \alg{G}$ be an arithmetic subgroup.

For every local field $L_v$ with $\mathrm{char}(L_v) = 0$,  every connected absolutely almost simple \(L_v\)-group \(\alg{H}\) and every homomorphism
$\varphi_v \colon \Gamma \to \alg{H}(L_v)$ such that $\overline{\varphi_v(\Gamma)}$ has non-empty interior, there are
\begin{itemize}
\item a place $w$ of $K$ and an isomorphism $j_v \colon L_v\to K_w$ of topological fields,
\item a homomorphism of algebraic groups over $j_v$
\[\eta_v\colon \alg{G} \times_K {K_w} \to \alg{H}, \]
\item a group homomorphism $\nu_v \colon \Gamma \to Z(\alg{H})(L_v)$
\end{itemize}
such that $\varphi_v(\gamma) = \nu_v(\gamma)\eta_v(\gamma)$ for all $\gamma \in \Gamma$.
In addition, $w, j_v, \eta_v$ and $\nu_v$ are uniquely determined by this condition.
\end{lemma}
\begin{proof}
By Lemma \ref{lemma:zariski-dense} \eqref{it:open-dense} the image of $\varphi_v$ is Zariski dense. We apply 
algebraic superrigidity to obtain a unique embedding $\sigma\colon K \to L_v$,
 a unique surjective $L_v$-homomorphism $\eta_v \colon \alg{G} \times_K L_v \to \alg{H}$ and a unique group homomorphism $\nu_v \colon \Gamma \to Z(\alg{H})(L)$ such that
$ \varphi_v(\gamma) = \nu_v(\gamma) \eta_v(\gamma)$ for all $\gamma \in \Gamma$.

Since $L_v$ is a complete field with an absolute value, the restriction of the absolute value to $K$ defines a unique place $w$ of $K$.
The closure of $K$ in $L_v$ is the completion $K_w$ of $K$ with respect to the induced absolute value.

Claim: $K_w = L_v$.\\[1ex]
 Assume for a contradiction that $[L_v:K_w] \geq 2$.
Since 
\[
	\Gamma \subseteq \alg{G}(K_w) \subsetneq \alg{G}(L_v)
\]
we deduce from Lemma \ref{lemma:zariski-dense} \eqref{it:dimension-not-open} that $\Gamma$ is nowhere dense in $\alg{G}(L_v)$.

\medskip

Intermediate step: \textit{The homomorphism $\eta_v \colon \alg{G}(L_v) \to \alg{H}(L_v)$ is closed and has finite central kernel $M$.}\\[1ex]
The kernel is finite since $\alg{G}$ is absolutely almost simple and $\eta_v$ is non-trivial (the image $\varphi_v(\Gamma)$ is infinite!). Since $L_v$ has characteristic zero, it follows that  $\eta_v$ is a central isogeny; see \cite[22.3]{Borel:alg-groups}. Since $\eta_v$ is $L_v$-analytic and the Lie $L_v$-algebra of $\alg{H}(L_v)$ is simple, it follows that $\eta_v$ is a submersion and therefore open; see \cite[I (2.5.4)]{Margulis:discrete-subgroups}. By \cite[I (2.3.4)]{Margulis:discrete-subgroups}, the image $\eta_v(\alg{G}(L_v))$ is closed. Thus
$\eta_v$ induces a homeomorphism between $\alg{G}(L_v)/M$ and a closed subgroup of $\alg{H}(L_v)$. The projection $\alg{G}(L_v) \to \alg{G}(L_v)/M$ is closed because $M$ is finite; this proves the intermediate step. 

\medskip

Contradiction: \textit{$\overline{\varphi_v(\Gamma)}$ has empty interior.}\\[1ex]
 The closed set $M \cdot \overline{\Gamma}$ (a finite union of nowhere dense sets) has empty interior in $\alg{G}(L_v)$.
Since $\eta_v$ is closed, we obtain
 	\[\eta_v^{-1}(\overline{\eta_v(\Gamma)}) \subseteq M \cdot \overline{\Gamma}\]
 and thus $\overline{\eta_v(\Gamma)}$ has empty interior. Again $\nu_v(\Gamma)\overline{\eta_v(\Gamma)}$
 is a closed set with empty interior, since it is a finite union of nowhere dense sets.
This contradicts the assumption that  the subset $\overline{\varphi_v(\Gamma)}$ has non-empty interior.
\end{proof}
\begin{proof}[Proof of Theorem \ref{thm:adelic-superrigidity}]
For every finite set $S$ of finite places of $L$, we denote by 
\[
	\mathrm{pr}_S\colon \bbA^f_L \to \prod_{v \in S} L_v
\]
the projection and by
\[\alg{H}(\mathrm{pr}_S)\colon \alg{H}(\bbA^f_L) \to \prod_{v \in S} \alg{H}(L_v)\]
 the corresponding projection on the groups of rational points. We define $\varphi_S := \alg{H}(\mathrm{pr}_S) \circ \varphi \colon \Gamma \to \prod_{v \in S} \alg{H}(L_v)$.
 
 \medskip

Step 1: \textit{For every $S \subseteq  V_f(L)$ the closure \(\overline{\varphi_v(\Gamma)} \subseteq \alg{H}(L_v)\) has a non-empty interior.
}\\[1ex]
Indeed, $\alg{H}(\pr_S)$ is open and $\alg{H}(\pr_S)\bigl(\overline{\varphi(\Gamma)}\bigr) \subseteq \overline{\varphi_S(\Gamma)}$, so  $\overline{\varphi_S(\Gamma)}$  contains a non-empty open set.

In particular, this applies to $S = \{v\}$. Local rigidity Lemma \ref{lem:local-rigidity} implies that for every finite place $v$ of $L$, there is a unique finite place $w(v)$ of $K$, a unique isomorphism $j_{v}\colon K_{w(v)} \to L_v$, a unique homomorphism
\[
	\eta_v\colon \alg{G} \times_K {K_{w(v)}} \to \alg{H}\times L_v
\]
over $j_v$, and a unique group homomorphism $\nu_v \colon \Gamma \to Z(\alg{H})(L_v)$ such that
\[
	\alg{H}(\mathrm{pr}_v) (\varphi(\gamma)) = \nu_v(\gamma)\eta_v(\gamma)
\]
holds for all $\gamma \in \Gamma$.

\medskip

Step 2: \emph{The map $V_f(L) \to V_f(K)$ with $v \mapsto w(v)$ is injective.}\\[1ex]
Assume for a contradiction that $w(v_1) = w(v_2)$ for two distinct finite places $v_1, v_2 \in V_f(L)$. Put $w = w(v_1) = w(v_2)$.
We know that the homomorphism $\varphi_{v_1,v_2}\colon \Gamma \to \alg{H}(L_{v_1}) \times \alg{H}(L_{v_2})$ is given by the formula
\[
	\gamma \mapsto (\nu_{v_1}(\gamma)\eta_{v_1}(\gamma),\nu_{v_2}(\gamma)\eta_{v_2}(\gamma))
\]
The image of the diagonal map $\alg{H}(K_w) \to \alg{H}(L_{v_1}) \times \alg{H}(L_{v_2})$ is nowhere dense (by Lemma \ref{lemma:zariski-dense} \eqref{it:dimension-not-open}) and therefore the image $J_\eta$ of the map $(\eta_{v_1}, \eta_{v_2}) \colon \alg{G}(K_w) \to \alg{H}(L_{v_1}) \times \alg{H}(L_{v_2})$ is nowhere dense. In addition, the image $J_\nu$ of the map $(\nu_{v_1},\nu_{v_2})\colon \Gamma \to \alg{H}(L_{v_1})  \times \alg{H}(L_{v_2})$ is finite.
By Baire's category theorem, the set $J_\nu J_\eta$ is nowhere dense and we deduce that the image of $\varphi_{v_1,v_2}$ is nowhere dense.
This contradicts Step~1 above.

\medskip

Step 3: \emph{$ [L:\bbQ] \leq [K:\bbQ]$ and in case of equality, the map $V_f(L) \to V_f(K)$ with $v \mapsto w(v)$ is surjective.}\\[1ex]
For every prime number $p$, the local degrees over $p$ add up to the global degree and we obtain
\[
	[L:\bbQ] = \sum_{v | p} [L_v : \bbQ_p] = \sum_{v | p } [K_{w(v)}:\bbQ_p] \leq [K:\bbQ].
\]
Clearly, if $[K:\bbQ] = [L:\bbQ]$ then it follows that all places of $K$ over $p$ occur on the right hand side.

\medskip

Observation: \textit{The kernel $\Gamma_\nu$ of the map $\nu: \Gamma \to \prod_{v\in V_f(L)}Z(\alg{H})(L_v) $ with $\gamma \mapsto (\nu_v(\gamma))_{v\in V_f(L)}$ has finite index.}\\[1ex] 
For every place $v \in V_f(L)$, the group $Z(\alg{H})(L_v)$ is a subgroup of the finite group $Z(\alg{H})(\overline{L_v}) \cong Z(\alg{H})(\bbC) =: Z$.
The arithmetic group $\Gamma$ is finitely generated and hence there are only finitely many homomorphisms from $\Gamma$ to $Z$.
The intersection of the kernels of all these homomorphisms has finite index in $\Gamma$ and is contained in $\Gamma_\nu$.

\medskip

Define $B = \prod_{v \in V_f(L)} L_v$. Observe that $B$ is an $L$-algebra and, in addition, a $K$-algebra identifying $L_v$ with $K_{w(v)}$ for all $v$ using the isomorphisms $j_v$. Lemma \ref{lemma:tuple-yields-morphism} implies the following

 \textit{There is a unique morphism of group schemes $\eta^0 \colon \alg{G}\times_K B \to \alg{H}\times_L B$ such that the following diagram commutes for all $v \in V_f(L)$:}
 \begin{equation*}
      \xymatrix{
        \alg{G}(B) \ar[d]_{\alg{G}(\pr_v)}\ar[r]^{\eta^0} & \alg{H}(B)\ar[d]^{\alg{H}(\pr_v)} \\
        \alg{G}(K_{w(v)}) \ar[r]^{\eta_v} & \alg{H}(L_v).
        }
      \end{equation*}
      
The arithmetic group $\Gamma_\nu \subseteq \alg{G}(K)$ is Zariski dense and 
$\eta(\Gamma_\nu) = \varphi(\Gamma_\nu) \subseteq \alg{H}(\bbA^f_L)$. We apply Lemma \ref{lemma:defined-over-subalgebra} 
with $E = \bbA^f_L$ (as above, this is an $L$- and a $K$-algebra) to see that $\eta$ is defined over $\bbA^f_L$, i.e., there is a morphism of group schemes
\[
	\eta^0\colon \alg{G} \times_K \bbA^f_L \to \alg{H}\times_L \bbA^f_L
\]
such that the diagram
 \begin{equation*}
      \xymatrix{
        \alg{G}(\bbA^f_L) \ar[d]_{\alg{G}(\pr_v)}\ar[r]^{\eta^0} & \alg{H}(\bbA^f_L)\ar[d]^{\alg{H}(\pr_v)} \\
        \alg{G}(K_{w(v)}) \ar[r]^{\eta_v} & \alg{H}(L_v)
        }
\end{equation*}
commutes for all $v \in V_f(L)$.

We combine the isomorphisms $j_v\colon L_v \to K_{w(v)}$ into an injective homomorphism 
$j \colon \bbA^f_L \to \bbA^f_K$ of topological rings which induces a homomorphisms of group schemes
\[
	\tilde{j} \colon \alg{G} \times_K \bbA^f_K \to \alg{G} \times_K \bbA^f_L.
\]
Composing $\eta^0$ with $\tilde{j}$ yields the required morphism
\[
	\eta \colon \alg{G} \times_K \bbA^f_K \to \alg{H}\times_L \bbA^f_L.
\]
By construction, we have
\[
	\varphi(\gamma) = \nu(\gamma)\eta(\gamma)
\]
and since $\varphi(\gamma),\eta(\gamma) \in \alg{H}(\bbA^f_L)$, it follows that 
\[
	\nu(\gamma) \in Z(\alg{H})(B) \cap \alg{H}(\bbA_L^f) = Z(\alg{H})(\bbA^f_L).
\]

\medskip

Finally, we discuss uniqueness.  Assume that $w'$, $j'_v$ (for $v \in V_f(L)$), $\eta'$ and $\nu'$ as before satisfy the condition $\varphi(\gamma) = \nu'(\gamma)\eta'(\gamma)$ for all $\gamma \in \Gamma$.
For all $v \in V_f(L)$, the homomorphism $\eta'$ induces a homomorphism
\[
	\eta'_v \colon \alg{G}\times_K K_{w'(v)} \to \alg{H}\times_L L_v
	\]
over $j'_v$ such that the diagramm
 \begin{equation*}
      \xymatrix{
        \alg{G}(\bbA^f_K) \ar[d]_{\alg{G}(\pr_v)}\ar[r]^{\eta'} & \alg{H}(\bbA^f_L)\ar[d]^{\alg{H}(\pr_v)} \\
        \alg{G}(K_{w'(v)}) \ar[r]^{\eta'_v} & \alg{H}(L_v),
        }
\end{equation*}
commutes. We obtain
\begin{align*}
	\varphi_v(\gamma) &= \alg{H}(\pr_v)\bigl(\nu'(\gamma)\eta'(\gamma)\bigr) = \alg{H}(\pr_v)(\nu'(\gamma)) \; \alg{H}(\pr_v)(\eta'(\gamma))\\
	&= \alg{H}(\pr_v)(\nu'(\gamma)) \; \eta'_v(\gamma)
\end{align*}
and the uniqueness statement in Lemma \ref{lem:local-rigidity} on local superrigidity implies $w'(v) = w(v)$, $j_v = j'_v$ and $\eta'_v = \eta_v$ for all $v \in V_f(L)$. The uniqueness in Lemma \ref{lemma:tuple-yields-morphism} implies that $\eta' = \eta$. Since 
$\varphi(\gamma) = \nu'(\gamma)\eta(\gamma) = \nu(\gamma)\eta(\gamma)$,
we conclude that $\nu' = \nu$. 
\end{proof}

As our main application of adelic superrigidity, we will now see that if two algebraically superrigid groups \(\mathbf{G}\) and \(\mathbf{H}\) have profinitely isomorphic arithmetic subgroups, then \(\mathbf{G}\) is adelically isomorphic to \(\mathbf{H}\).  To give the precise statement, we recall that as in \cite{Kammeyer-Kionke:rigidity}*{(A.1) and (A.2)}, an arithmetic subgroup \(\Gamma \subset \mathbf{G}(K)\) comes with a canonical projection \(q_{\mathbf{G}} \colon \widehat{\Gamma} \rightarrow \overline{\Gamma} \subseteq \mathbf{G}(\mathbb{A}^f_K)\) from the profinite completion to the \emph{congruence completion} and the latter is an open subgroup of the finite adele points of \(\mathbf{G}\) by strong approximation.

\begin{theorem}\label{thm:profinite-isomorphism-adelic}
Let $\alg{G}$ and $\alg{H}$ be absolutely almost simple algebraically superrigid algebraic groups defined over number fields $K$ and $L$, respectively.  Let $\Gamma \subseteq \alg{G}(K)$ and $\Delta \subseteq \alg{H}(L)$ be arithmetic subgroups.

Assume that $\Psi \colon \widehat{\Gamma} \to \widehat{\Delta}$ is an isomorphism of profinite completions. Then there are
\begin{itemize}
\item an isomorphism $j_1 \colon \bbA^f_L \to \bbA_K^f$ of topological rings,
\item an isomorphism of group schemes $\eta_1 \colon \alg{G} \times_K \bbA_K^f \to \alg{H} \times_L \bbA_L^f$ over $j_1$,
\item a homomorphism $\hat{\nu}_1 \colon \widehat{\Gamma} \to Z(\alg{H})(\bbA^f_L)$ with finite image
\end{itemize}
such that $q_{\alg{H}}(\Psi(\hat{\gamma}))  = \nu_1(\hat{\gamma})\eta_1(q_{\alg{G}}(\hat{\gamma}))$ for all $\hat{\gamma} \in \widehat{\Gamma}$.

In particular, there is an open normal subgroup $U \normal_o \widehat{\Gamma}$ such that
 \begin{equation*}
      \xymatrix{
      	U \ar[d]_{q_\alg{G}}\ar[r]^{\cong}_{\Psi} & \Psi(U) \ar[d]_{q_{\alg{H}}}\\
        \alg{G}(\bbA^f_K) \ar[r]^{\cong}_{\eta_1} & \alg{H}(\bbA^f_L),
        }
\end{equation*}
commutes, i.e., the finite index subgroups $\Gamma' = \Gamma \cap U$ and $\Delta' = \Delta \cap \Psi(U)$ have isomorphic profinite completions and isomorphic pro-congruence completions.
\end{theorem}
\begin{proof}
Define $\varphi_1 \colon \Gamma \to \alg{H}(\bbA_L^f)$ to be the composition of $\Psi$ with the projection $q_{\alg{H}}\colon \widehat{\Delta} \to \overline{\Delta} \subseteq \alg{H}(\bbA_L^f)$ onto the congruence completion. 
Similarly, define  $\varphi_2 \colon \Delta \to \alg{G}(\bbA_K^f)$ to be the composition of $\Psi^{-1}$ with the projection $q_{\alg{G}}\colon\widehat{\Gamma} \to \overline{\Gamma} \subseteq \alg{G}(\bbA_K^f)$. 

Adelic superrigidity \ref{thm:adelic-superrigidity} provides maps, resp. morphisms
\begin{align*}
&w_1 \colon V_f(L) \to V_f(K), \quad &&w_2 \colon V_f(K) \to V_f(L)\\
&j_{1,v} \colon L_v \to K_{w_1(v)},\quad &&j_{2,u}\colon K_u \to L_{w_2(u)}\\ 
&\eta_1\colon \alg{G}\times_K \bbA^f_K \to \alg{H}\times_L \bbA^f_L, \quad
&&\eta_2\colon \alg{H}\times_L \bbA^f_L \to \alg{G}\times_K \bbA^f_K\\
&\nu_1 \colon \Gamma \to Z(\alg{H})(\bbA^f_L), \quad
&&\nu_2 \colon \Delta \to Z(\alg{G})(\bbA^f_K). 
\end{align*}
such that
\[
	\varphi_1(\gamma) = \nu_1(\gamma) \eta_1(\gamma), \quad \varphi_2(\delta) = \nu_2(\delta) \eta_2(\delta)
\]
for all $\gamma \in \Gamma$ and $\delta \in \Delta$. The maps $\nu_1$ and $\nu_2$ have finite images.
Let $\hat{\nu_1} \colon \widehat{\Gamma} \to Z(\alg{H})(\bbA^f_L)$ denote homomorphism induced from $\nu_1$ on the profinite completion.
Since $\Gamma$ is dense in $\widehat{\Gamma}$, this proves the asserted identity 
\[q_{\alg{H}} \circ \Psi(\hat{\gamma})  = \nu_1(\hat{\gamma})\eta_1(q_{\alg{G}}(\hat{\gamma})).\]

From adelic superrigidity we know further that $[K:\bbQ] = [L : \bbQ]$ and that
$w_1, w_2$ are bijections. Let 
\[
	U = \ker(\hat{\nu}_1) \cap \ker(\hat{\nu}_2 \circ \Psi) \normal_o \widehat{\Gamma}.
\]
We define $\Gamma' = U \cap \Gamma$ and $\Delta' = \Psi(U) \cap \Delta$.
We show that the maps $w_*$, the isomorphisms $j_*$, and the morphisms $\eta_*$ are inverses of each other.

Define $w = w_1 \circ w_2$, $\eta = \eta_2 \circ \eta_1$ and $j_u = j_{1,w_2(u)} \circ j_{2,u}$. Then $\eta$ is a homomorphism of
group schemes $\alg{G}\times_K \bbA^f_K \to \alg{G} \times_K \bbA^f_K$
over the map $j\colon \bbA_K^f \to \bbA_K^f$ induced from the collection of isomorphisms $j_u$.
We will prove the following\\[1ex]
Claim: \textit{$\eta(\gamma) = \gamma$ for all $\gamma \in \Gamma'$}.\\[1ex] This will complete the proof, since the uniqueness statement in the adelic superrigidity theorem (applied to $\Gamma'$) implies that $w = \id_{V_f(K)}$, $j = \id_{\bbA_K^f}$ and $\eta = \id$. Interchanging the roles of $\alg{G}$ and $\alg{H}$, we deduce that $j_1$ and $\eta_1$ are isomorphisms.
The commutativity of the diagram is immediate from the continuity of all involved homomorphisms and the density of $\Gamma'$ in $U$.

\medskip

\textit{Proof of Claim:} Let $\gamma \in \Gamma'$. The group $\Delta'$ is dense in $\Psi(U)$, so that we can find a sequence $(\delta_n)_{n \in \bbN}$ in $\Delta'$ which converges to $\Psi(\gamma)$. Then we obtain
\begin{align*}
	\eta(\gamma) &= \eta_2(\eta_1(\gamma)) = \eta_2(q_{\alg{H}}(\Psi(\gamma)))\\ &= \lim_{n \to \infty} \eta_2(\delta_n)
	= \lim_{n \to \infty} q_{\alg{G}}(\Psi^{-1}(\delta_n)) = \gamma
\end{align*}
using that $\Psi^{-1}(\delta_n)$ converges to $\gamma$ in $U$.
\end{proof}

For future reference, let us close this section with a useful corollary.

\begin{corollary}
Let $K$ be an algebraic number field and let
$\alg{G}$ be an absolutely almost simple algebraically superrigid algebraic $K$-group.
Let $\Gamma \subseteq \alg{G}(K)$ be an arithmetic subgroup.
The kernel $\widetilde{C}$ of the canonical homomorphism 
\[
	\widehat{\Gamma} \to \alg{G}(\bbA^f_K)/Z(\alg{G})(\bbA^f_K)
\]
is a characteristic subgroup of $\widehat{\Gamma}$.
\end{corollary}
\begin{proof}
Let $\Psi \in \Aut(\widehat{\Gamma})$ and let $\hat{\gamma} \in \widehat{\Gamma}$ . We apply \ref{thm:profinite-isomorphism-adelic} to $\Psi$ and obtain
\[
	q_{\alg{G}}(\Psi(\hat{\gamma})) \in \hat{\nu}_1(\hat{\gamma})\eta_1(q_{\alg{G}}(\hat{\gamma})).
\]
with $\hat{\nu}_1(\hat{\gamma}) \in Z(\alg{G})(\bbA_K^f)$.
Suppose $\hat{\gamma} \in \widetilde{C}$, then $q_{\alg{G}}(\hat{\gamma}) \in Z(\alg{G})(\bbA^f_K)$ and (since $\eta_1$ is an isomorphism) we conclude that  $\eta_1(q_{\alg{G}}(\hat{\gamma})) \in Z(\alg{G}(\bbA_K^f)) = Z(\alg{G})(\bbA^f_K)$.
\end{proof}

\section{Profinite commensurability classification} \label{sec:classification}

In this section, we will prove Theorems~\ref{thm:commensurability-classification} and~\ref{thm:profinite-commensurability-classification} and hence establish the profinite commensurability classification of arithmetic subgroups of algebraic groups with CSP.  We start by recalling Margulis arithmeticity to see how it motivates the introduction of \(\mathbf{G}\)-arithmetic pairs in Definition~\ref{def:g-arithmetic-pair}.

Let \(\K\) be either the field of real or complex numbers and let \(\mathbf{G}\) be a connected absolutely almost simple linear algebraic \(\K\)-group such that \(\mathrm{rank}_\K \mathbf{G} \ge 2\).  Then by \cite{Margulis:discrete-subgroups}*{Theorem IX.1.11, p.\,298 and (\(\ast \ast\)) on p.\,293/294}, any lattice \(\Gamma \le \mathbf{G}(\K)\) is \emph{arithmetic}:  There exists a subfield \(K \subset \mathbb{K}\) of finite degree over \(\mathbb{Q}\), a connected simply connected absolutely almost simple linear algebraic \(K\)-group \(\mathbf{H}\), and a continuous homomorphism \(\varphi \colon \mathbf{H}(\K) \rightarrow \mathbf{G}(\K)\) such that
\begin{enumerate}[(i)]
  \item \label{item:dense} \(K\) is dense in \(\mathbb{K}\),
\item \label{item:finite-kernel-cokernel} the homomorphism \(\varphi\) has finite kernel and the image of \(\varphi\) contains the unit component \(\mathbf{G}(\K)^0\),
\item \label{item:commensurable} the subgroups \(\varphi(\mathbf{H}(\mathcal{O}_K))\) and \(\Gamma\) of \(\mathbf{G}(\K)\) are commensurable,
\item \label{item:compact} all the other infinite places \(\sigma\) of \(K\) are real and define an anisotropic \(\R\)-group \(\mathbf{H} \times_\sigma \mathbb{R}\).
\end{enumerate}
Here the group \(\mathbf{H}(\mathcal{O}_K)\) of \(K\)-integral points is defined by means of a fixed \(K\)-embedding \(\mathbf{H} \subset \mathbf{GL_n}\).  Any other choice of embedding yields an (abstractly) commensurable group, so the commensurability class of \(\mathbf{H}(\mathcal{O}_K)\) is well-defined and coincides with that of \(\Gamma\) by~\eqref{item:commensurable}.  By~\eqref{item:finite-kernel-cokernel} and~\eqref{item:commensurable}, there exists a finite index subgroup \(\Lambda \subset \mathbf{H}(\mathcal{O}_K)\) which embeds as a finite index subgroup of \(\Gamma\) via \(\varphi\).  If \(K' \subset \mathbb{K}\), \(\mathbf{H}' \subset \mathbf{GL_n}\), \(\varphi'\), and \(\Lambda'\) are another choice as above, then the finite index subgroup \(\Lambda_0 = \varphi_{|\Lambda}^{-1}(\varphi(\Lambda) \cap \varphi'(\Lambda'))\) of~\(\Lambda\) embeds via \(\delta = \varphi_{|\Lambda'}^{'-1} \circ \varphi\) into a finite index subgroup of \(\mathbf{H}'(\mathcal{O}_{K'})\).  Hence \(\delta(\Lambda_0)\) is Zariski dense in \(\mathbf{H}'(K')\).  By~\eqref{item:finite-kernel-cokernel} and~\eqref{item:compact}, we have \(\mathrm{rank}_S \mathbf{H} = \mathrm{rank}_\K \mathbf{G} \ge 2\) so \(\mathbf{H}\) exhibits algebraic superrigidity as in Definition~\ref{def:superrigidity}: there exists a unique field embedding \(\sigma \colon K \rightarrow K'\) and a unique \(K'\)-epimorphism \(\eta \colon \mathbf{H} \times_\sigma K' \rightarrow \mathbf{H}'\) with the property that \(\eta\) extends \(\delta\) (up to multiplication with central elements).  Since \(\mathbf{H}\) is absolutely almost simple and \(\mathbf{H'}\) is simply connected, the epimorphism \(\eta\) is in fact a \(K'\)-isomorphism.  By symmetry, we also have a homomorphism \(K' \rightarrow K\), so \(\sigma\) is an isomorphism as well.

If \(\nu\) and \(\nu'\) denote the infinite places of \(K\) and \(K'\) corresponding to the given embeddings \(K \subset \K\) and \(K' \subset \K\), the \(K'\)-isomorphism \(\eta\) induces the \(\K\)-group isomorphism
\[ \mathbf{H}'\times_{\nu \circ \sigma^{-1}} \K \cong (\mathbf{H} \times_\sigma K') \times_{\nu \circ \sigma^{-1}} \K \cong \mathbf{H} \times_\nu \K, \]
so \(\nu \circ \sigma^{-1}\) is the unique infinite place of \(K'\) at which \(\mathbf{H}'\) is isotropic.  This shows that either \(\nu \circ \sigma^{-1} = \nu'\), hence \(K = K'\) as subfields of \(\mathbb{K}\) and \(\sigma = \mathrm{id}_K\), or potentially \(K\) and \(K'\) are complex conjugates and \(\sigma\) is complex conjugation if \(\mathbb{K} = \C\).  So the lattice \(\Gamma\) determines the subfield \(K \subset \mathbb{K}\) uniquely up to complex conjugation, justifying the notation \(K = K_\Gamma\).  Moreover, once the subfield $K\subset \mathbb{K}$ is fixed, the group \(\mathbf{H} = \mathbf{H}_\Gamma\) is uniquely determined by \(\Gamma\) up to \(K\)-isomorphism.  We are now ready to prove Theorem~\ref{thm:commensurability-classification}.  More precisely, we show the following statement.

\begin{theorem} \label{thm:inverse}
  Assigning \(\Gamma \mapsto (K_\Gamma, \mathbf{H}_\Gamma)\) defines a map from commensurability classes of lattices in \(\mathbf{G}(\K)\) to isomorphism classes of \(\mathbf{G}\)-arithmetic pairs.  It is the inverse of the map \(\Phi\) from Theorem~\ref{thm:commensurability-classification}.
\end{theorem}

\label{page:g-arithmetic-pairs-iso}  Here, two \(\mathbf{G}\)-arithmetic pairs \((K_1, \mathbf{H_1})\) and \((K_2, \mathbf{H_2})\) are \emph{isomorphic} if \(K_1 = K_2\) (or if \(K_1\) is the complex conjugate of \(K_2\)) and if \(\mathbf{H_1}\) is \(K_1\)-isomorphic to \(\mathbf{H_2}\) (or to \(\mathbf{H_2} \times_\sigma K_1\) where \(\sigma\) is complex conjugation).

\begin{proof}
  To see well-definedness of the assignment,  it remains to show that \(\mathbf{H}_\Gamma\) is \(\K\)-isogenous to \(\mathbf{G}\).  This follows from~\eqref{item:finite-kernel-cokernel} because the homomorphism \(\varphi\) induces a \(\K\)-Lie algebra isomorphism \(\beta \colon \mathrm{Lie}(\mathbf{H}_\Gamma) \xrightarrow{\cong} \mathrm{Lie}(\mathbf{G})\) which by~\cite{Margulis:discrete-subgroups}*{Proposition~I.1.4.13} is the differential of a \(\K\)-isogeny \(\pi \colon \mathbf{H}_\Gamma \rightarrow \mathbf{G}\).

If two lattices \(\Gamma, \Lambda \le \mathbf{G}(\K)\) are commensurable, a straightforward refinement of the above discussion gives that the \(\mathbf{G}\)-arithmetic pairs \((K_\Gamma, \mathbf{H}_\Gamma)\) and \((K_\Lambda, \mathbf{H}_\Lambda)\) are isomorphic.  Hence the assignment descends to a map on commensurability classes which is a right inverse to \(\Phi\):  \(\Phi(K_\Gamma, \mathbf{H}_\Gamma)\) is commensurable with \(\Gamma\) by~\eqref{item:commensurable}.  But it is also a left inverse of \(\Phi\): for a given pair \((K, \mathbf{H})\), let \(\eta \colon \mathbf{H} \rightarrow \mathbf{G}\) be a \(\K\)-isogeny .  Then \(\varphi = \eta_\K\) is a continuous homomorphism \(\varphi \colon \mathbf{H}(\K) \rightarrow \mathbf{G}(\K)\) with finite kernel whose image contains \(\mathbf{G}(\K)^0\).  It embeds a finite index subgroup of \(\mathbf{H}(\mathcal{O}_K)\) as a lattice \(\Gamma = \Gamma_{(K, \mathbf{H})}\) of \(\mathbf{G}(\K)\).  By the uniqueness discussion above the theorem, we have \((K_\Gamma, \mathbf{H}_\Gamma) \cong (K, \mathbf{H})\). 
\end{proof}

The following theorem reveals that the necessary conditions on a number field \(K\) to occur in a \(\mathbf{G}\)-arithmetic pair are actually also sufficient.  Hence the Lie group \(\mathbf{G}(\K)\) always comes with a multitude of distinct commensurability classes of lattices.

\begin{theorem} \label{thm:existence-of-G-arithmetic-pairs}
  Let \(K \subset \K\) be any dense subfield with \([K : \Q] < \infty\) and such that all other infinite places are real.  Then there exists a \(\mathbf{G}\)-arithmetic pair \((K, \mathbf{H})\).
\end{theorem}

\begin{proof}
  If \(\K = \C\), let \(\mathfrak{g}\) be the simple complex Lie algebra of \(\mathbf{G}\).  We fix a compact real form \(\mathfrak{u} \subset \mathfrak{g}\).  Starting with a Chevalley basis of \(\mathfrak{g}\) adapted to \(\mathfrak{u}\), one easily obtains a basis of \(\mathfrak{u}\) whose structure constants are rational integers~\cite{Borel:Clifford}*{Equation~(23)}.  The \(\Q\)-span of this basis thus defines a \(\Q\)-form \(\mathfrak{u}_\Q\) of \(\mathfrak{u}\), hence the connected component of the automorphism group \(\mathrm{Aut}_\Q(\mathfrak{u}_\Q)^0\) is a connected absolutely simple \mbox{\(\Q\)-group} whose Lie algebra is \(\C\)-isomorphic to \(\mathfrak{g}\).  Passing to the universal covering group and extending scalars, we thus obtain a connected simply connected absolutely almost simple \(K\)-group \(\mathbf{H}\) which is \(\C\)-isogenous to \(\mathbf{G}\) and \(\R\)-anisotropic at all other infinite places of \(K\).
  
If \(\K = \R\), then \(K \subset \R\) is totally real and Borel constructs in~\cite{Borel:Clifford}*{Proposition~3.8} a \(K\)-form \(\mathfrak{g}_K\) of the simple real Lie algebra \(\mathfrak{g}\) of \(\mathbf{G}\) such that all conjugates \({}^\nu \mathfrak{g}_K\) are \(\R\)-isomorphic to a compact form of \(\mathfrak{g} \otimes \C\).  Again, the universal covering of the unit component of the automorphism group of \(\mathfrak{g}_K\) is a \(K\)-group \(\mathbf{H}\) as required.  Note that Borel excludes the case \(K = \Q\) in his statement only for a later purpose.  In fact, \cite{Borel:Clifford}*{Proposition~3.7} shows that there exists a \(\Q\)-form of \(\mathfrak{g}\) which is all we need in this case.
\end{proof}

It is easy to see that commensurable groups are profinitely commensurable~\cite{Kammeyer:commensurability}*{Proposition~17}.  We shall next see that the corresponding weakening of equivalence of \(\mathbf{G}\)-arithmetic pairs is captured in the following definition.

\begin{definition} \label{def:locally-isomorphic}
  Two \(\mathbf{G}\)-arithmetic pairs \((K_1, \mathbf{H_1})\) and \((K_2, \mathbf{H_2})\) are called \emph{locally isomorphic} if there exists an isomorphism \(j \colon \mathbb{A}^f_{K_2} \xrightarrow{\cong} \mathbb{A}^f_{K_1}\) of topological rings and a \(j\)-group scheme isomorphism
  \[ \mathbf{H_1} \times_{K_1} \mathbb{A}^f_{K_1} \xrightarrow{\ \cong \ } \mathbf{H_2} \times_{K_2} \mathbb{A}^f_{K_2}. \]
\end{definition}

Moreover, we call two number fields \emph{locally isomorphic} if there exists an isomorphism of the finite adele rings as in the definition.

\begin{theorem} \label{thm:profinnitely-commensurable-iff-locally-isomorphic}
Suppose \(\mathbf{G}\) has CSP and let \(\Gamma, \Lambda \le \mathbf{G}(\K)\) be lattices.  Then the following are equivalent.
  \begin{enumerate}[(i)]
  \item \label{item:profinitely-commensurable} The lattice \(\Gamma\) is profinitely commensurable with \(\Lambda\).
  \item \label{item:local-isomorphism} The \(\mathbf{G}\)-arithmetic pair \((K_\Gamma, \mathbf{H}_\Gamma)\) is locally isomorphic to \((K_\Lambda, \mathbf{H}_\Lambda)\).
\end{enumerate}
\end{theorem}

\begin{proof}
  \emph{\eqref{item:local-isomorphism} implies \eqref{item:profinitely-commensurable}:} As in~\cite{Kammeyer-Kionke:rigidity}*{(A.2)}, we have a short exact sequence
  \[ 1 \longrightarrow C(K_\Gamma, \mathbf{H}_\Gamma) \longrightarrow \widehat{\mathbf{H}_\Gamma(\mathcal{O}_{K_\Gamma})} \xrightarrow{\ \chi \ } \prod_{v \in V^f(K_\Gamma)} \mathbf{H}_\Gamma(\mathcal{O}_v) \longrightarrow 1 \]
  of profinite groups and similarly for \((K_\Lambda, \mathbf{H}_\Lambda)\).  Here \(\mathcal{O}_v\) is the valuation ring of the local completion of \(K_\Gamma\) at the finite place \(v \in V^f(K_\Gamma)\) and \(\chi\) is the canonical extension of the diagonal embedding of \(\mathbf{H}_\Gamma(\mathcal{O}_{K_\Gamma})\) to the profinite completion.  The kernel \(C(K_\Gamma, \mathbf{H}_\Gamma)\) of \(\chi\) is just the \emph{congruence kernel} which is finite by assumption.  Hence we find a finite index subgroup \(\Gamma_0 \subset \mathbf{H}_\Gamma(\mathcal{O}_{K_\Gamma})\) which intersects \(C(K_\Gamma, \mathbf{H}_\Gamma)\) trivially and which is also isomorphic to a finite index subgroup of \(\Gamma\).  Strong approximation, which holds true for \(\mathbf{H}_\Gamma\) by \cite{Platonov-Rapinchuk:algebraic-groups}*{Theorem~7.12, p.\,427}, implies that \(\chi\) embeds the closure of \(\Gamma_0\) in \(\widehat{\mathbf{H}_\Gamma(\mathcal{O}_{K_\Gamma})}\) as an open subgroup of \(\prod_{v \in V^f(K_\Gamma)} \mathbf{H}_\Gamma(\mathcal{O}_v)\) and hence also of \(\mathbf{H}_\Gamma(\mathbb{A}^f_{K_\Gamma})\).  The latter is isomorphic to \(\mathbf{H}_\Lambda(\mathbb{A}^f_{K_\Lambda})\) by~\eqref{item:local-isomorphism}.   Applying the argument backwards for \(\Lambda\) and \((K_\Lambda, \mathbf{H}_\Lambda)\), we obtain that \(\Gamma\) is profinitely commensurable with \(\Lambda\).

  \emph{\eqref{item:profinitely-commensurable} implies \eqref{item:local-isomorphism}:} this is an immediate consequence of adelic superrigidity, Theorem~\ref{thm:profinite-isomorphism-adelic}, applied to profinitely isomorphic finite index subgroups \(\Gamma_0 \subset \Gamma\) and \(\Lambda_0 \subset \Lambda\).
\end{proof}

\begin{proof}[Proof of Theorem~\ref{thm:profinite-commensurability-classification}.]
Theorems~\ref{thm:commensurability-classification} and the ``\eqref{item:local-isomorphism} implies \eqref{item:profinitely-commensurable}''-part of Theorem~\ref{thm:profinnitely-commensurable-iff-locally-isomorphic} show that \(\Phi\) descends to a well-defined surjection on local isomorphism classes of \(\mathbf{G}\)-arithmetic pairs.  The ``\eqref{item:profinitely-commensurable} implies \eqref{item:local-isomorphism}''-part of Theorem~\ref{thm:profinnitely-commensurable-iff-locally-isomorphic} combined with Theorem~\ref{thm:inverse} show that this map is injective.
\end{proof}

\section{Consequences on profinite rigidity} \label{sec:profinite-rigidity}

In this section we prove the consequences of the profinite commensurability classification on profinite rigidity that we stated in the introduction, starting with Theorem~\ref{thm:exceptional-cases}.  As we will explain in a moment, it is a consequence of noncommutative Galois cohomology computations by M.\,Kneser that for the exceptional Cartan--Killing types under consideration, the \(K\)-form of a group is determined by the \(K_v\)-forms at the infinite places \(v\) of~\(K\).  This has the striking consequence that for these types, both the commensurability and the profinite commensurability classification of lattices depends on the number field \(K_{\Gamma} \subset \K\) only. 

\begin{proof}[Proof of Theorem~\ref{thm:exceptional-cases}.]
  By Theorem~\ref{thm:commensurability-classification} (and Theorem~\ref{thm:inverse}), commensurability classes of lattices \(\Gamma \le \mathbf{G}(\R)\) are in 1-1--correspondence with isomorphism classes \((K_\Gamma, \mathbf{H}_\Gamma)\) of \(\mathbf{G}\)-arithmetic pairs.  By Theorem~\ref{thm:existence-of-G-arithmetic-pairs}, for every totally real number subfield \(K \subset \mathbb{R}\), there exists a \(\mathbf{G}\)-arithmetic pair \((K, \mathbf{H})\).  For assertion~\eqref{item:comm-subfields} of the theorem, it thus remains to show that given \(K \subset \mathbb{R}\), the isomorphism type of the \(\mathbf{G}\)-arithmetic pair \((K, \mathbf{H})\) is unique.

  Since two such groups \(\mathbf{H}\), \(\mathbf{H}'\) are simply connected and \(\R\)-isogenous to~\(\mathbf{G}\) (at the distinguished infinite place), the groups are \(\R\)-isomorphic and hence in particular \(\C\)-isomorphic.  This shows \(\mathbf{H}\) and \(\mathbf{H}'\) are \(K\)-forms of one another.  The various \(K\)-forms of \(\mathbf{H}\) are classified up to \(K\)-isomorphism by the first noncommutative Galois cohomology set \(H^1(K, \Aut_{\overline{K}}(\mathbf{H}))\).  The Dynkin diagram of \(\mathbf{G}\) and hence of \(\mathbf{H}\) has no symmetries, which has the effect that \(\Aut_{\overline{K}}(\mathbf{H}) \cong \overline{\mathbf{H}}\) is isomorphic to the adjoint group of \(\mathbf{H}\).  Moreover, for the types \(E_8\), \(F_4\), and \(G_2\), weight lattice and root lattice coincide which has the consequence that \(\mathbf{H}\) has trivial center~\cite{Platonov-Rapinchuk:algebraic-groups}*{table on p.\,64}, so \(\overline{\mathbf{H}} = \mathbf{H}\).  One of the main theorems in Galois cohomology of algebraic groups states that the diagonal map
  \[ \theta \colon H^1(K, \mathbf{H}) \longrightarrow \prod_{v \in V_\infty(K)} H^1(K_v, \mathbf{H}) \]
  is bijective whenever \(\mathbf{H}\) is a simply connected group~\cite{Platonov-Rapinchuk:algebraic-groups}*{Theorem~6.6}.  For a \(\overline{K}\)-isomorphism \(f \colon \mathbf{H} \xrightarrow{\cong} \mathbf{H}'\), let \(a_\sigma = f^{-1} f^\sigma\) be the \(1\)-cocycle on \(\mathrm{Gal}(\overline{K}/K)\) with values in \(\Aut_{\overline{K}}(\mathbf{H}) \cong \mathbf{H}\).  Since \(\R\)-anisotropic forms are unique, the condition that both \((K, \mathbf{H})\) and \((K, \mathbf{H}')\) are \(\mathbf{G}\)-arithmetic pairs precisely says that \(a_\sigma\) maps to the distinguished element in \(H^1(K_v, \mathbf{H})\) for all \(v \in V_\infty(K)\).  Since \(\theta\) is a bijection of pointed sets, this implies that \(\mathbf{H}\) and \(\mathbf{H}'\) are \(K\)-isomorphic.

    The group \(\mathbf{G}\) has CSP by~\cite{Platonov-Rapinchuk:algebraic-groups}*{Theorem~9.24}.  Hence by Theorem~\ref{thm:profinite-commensurability-classification}, profinite commmensurability classes of lattices in \(\mathbf{G}(\R)\) correspond to local equivalence classes of \(\mathbf{G}\)-arithmetic pairs.  For assertion~\eqref{item:prof-comm-adelic} of the theorem, it thus remains to show that for any two \(\mathbf{G}\)-arithmetic pairs \((K, \mathbf{H})\) and \((K', \mathbf{H}')\) such that there exists an isomorphism \(j \colon \mathbb{A}^f_{K'} \xrightarrow{\cong} \mathbb{A}^f_{K}\), we have a \(j\)-group scheme isomorphism
  \[ \mathbf{H} \times_K \mathbb{A}^f_K \xrightarrow{\ \cong \ } \mathbf{H}' \times_{K'} \mathbb{A}^f_{K'}. \]
  The isomorphism~\(j\) yields a bijection \(v \mapsto v'\) between the finite places of \(K\) and \(K'\) as well as isomorphisms of the corresponding completions; cf.~\cite[(2.5), p.~238]{Klingen:similarities}.  Similarly as above, for every finite place \(v\) of \(K\), the \(K_v\)-forms of \(\mathbf{H}\) are classified by the first Galois cohomology \(H^1(K_v, \mathbf{H})\) which is trivial by~\cite{Kneser:galois}*{Satz~1}.  Thus there are isomorphisms \(f_v \colon \mathbf{H} \cong \mathbf{H}'\) defined over \(K_v \cong K_{v'}\) for all finite places \(v\). We would like to assemble these isomorphisms to an isomorphism over the rings of finite adeles. To this end we need to choose the isomorphisms carefully.
  
  We fix integral models for $\mathbf{H}$ and $\mathbf{H}'$ over the rings of algebraic integers $\mathcal{O}_K$ and $\mathcal{O}_{K'}$ respectively. There is a finite set of prime numbers $S \subseteq \bbZ$ such that over
   $S^{-1}\mathcal{O}_K$ and $S^{-1}\mathcal{O}_{K'}$ the group schemes $\mathbf{H}$ and $\mathbf{H}'$ are smooth and semi-simple of same type everywhere locally (in the sense of SGA 3); see \cite[Exp.~XIX, Thm.~2.5]{SGA3-3}.
Groups of the exceptional types considered here are split over all finite fields and hence for each finite place $v$ of $K$ which does not lie over a prime in $S$ we can find isomorphisms
\[
	\mathbf{H} \times_{\mathcal{O}_K} k_v \xrightarrow{\ \cong \ }  \mathbf{H}' \times_{\mathcal{O}_{K'}} {k'}_{v'}
\]
defined over the induced isomorphism $k'_{v'} \to k_{v}$ between the finite residue fields.
By \cite[Exp.~XXIV, Cor.~1.12]{SGA3-3} these isomorphisms lift to isomorphisms
\[
	f^0_v\colon \mathbf{H} \times_{\mathcal{O}_K} \mathcal{O}_{K,v} \xrightarrow{\ \cong \ }  \mathbf{H}' \times_{\mathcal{O}_{K'}}\mathcal{O}_{K',v'}
\]
over the associated local rings and via base change provide isomorphisms $f_v \colon \mathbf{H} \times_{\mathcal{O}_K} K_v \xrightarrow{\ \cong \ }  \mathbf{H}' \times_{\mathcal{O}_{K'}} K'_{v'}$. For every finite place $v$ which divides a prime in $S$, we choose an arbitrary isomorphism
$f_v$.
  
    Assembling the isomorphisms $(f_v)_{v\in V_f}$, Lemma~\ref{lemma:tuple-yields-morphism} yields an isomorphism \(f \colon \mathbf{H} \xrightarrow{\cong} \mathbf{H}'\) over \(B = \prod_{v \nmid \infty} K_v\).  Since \(\mathbf{H}(K)\) is Zariski dense and by construction \(f(\mathbf{H}(K)) \subset \mathbf{H}'(\mathbb{A}^f_{K'})\), the isomorphism \(f\) is defined over \(\mathbb{A}^f_K\) by Lemma~\ref{lemma:defined-over-subalgebra}.
\end{proof}

\begin{example} \label{example:non-cocompact-q}
For \(\mathbf{G}\) as in the theorem, the commensurability class of all lattices \(\Gamma\) with \(K_\Gamma = \Q\) is precisely the set of all non-cocompact lattices in \(\mathbf{G}(\R)\).  Indeed,  the group \(\mathbf{H}_\Gamma\) in a \(\mathbf{G}\)-arithmetic pair \((K_\Gamma, \mathbf{H}_\Gamma)\) of a non-cocompact lattice \(\Gamma \le \mathbf{G}(\R)\) is \(K_\Gamma\)-isotropic.  Hence~\eqref{item:compact} in the definition of arithmetic lattices implies that \(K_\Gamma\) has only one infinite place which is real because \(\mathbb{K} = \R\).  Hence \(K_\Gamma = \Q\).  Conversely, for \(\mathbf{G}\) of the above exceptional types, we just saw that the \(K_\Gamma\)-group \(\mathbf{H}_\Gamma\) in a \(\mathbf{G}\)-arithmetic pair is unique up to \(K_\Gamma\)-isomorphism.  Hence if \(K_\Gamma = \Q\), then \(\mathbf{H}_\Gamma\) is the unique simply connected \(\Q\)-form of \(\mathbf{G}\) with \(\mathrm{rank}_\Q \mathbf{H} = \mathrm{rank}_\R \mathbf{G}\).  This \(\Q\)-form can for instance be constructed from a basis of the Lie algebra \(\mathfrak{g}\) of \(\mathbf{G}(\R)\) with rational structure constants which is an extension of a basis of a maximal abelian subalgebra \(\mathfrak{a} \subset \mathfrak{p}\) in a Cartan decomposition \(\mathfrak{g} \cong \mathfrak{k} \oplus \mathfrak{p}\)~\cite{Kammeyer:rational}*{Theorem~4.1}.  So \(\mathbf{H}_\Gamma\) is \(\Q\)-isotropic, hence \(\Gamma\) is non-cocompact.
\end{example}

The lack of Dynkin diagram symmetries was an important point in the above proof of Theorem~\ref{thm:exceptional-cases}.  It is also the reason why we have to exclude types \(A_n\), \(D_n\), and \(E_6\) from the conclusion of Theorem~\ref{thm:rational-field-solitary} because we need a trivial outer automorphism group to guarantee that the Hasse principle does not fail.

\begin{proof}[Proof of Theorem~\ref{thm:rational-field-solitary}]
Let \((K, \mathbf{H}')\) be a \(\mathbf{G}\)-arithmetic pair such that \(\Gamma_{(K, \mathbf{H}')}\) is profinitely commensurable with \(\Gamma_{(\Q, \mathbf{H})}\).  Then finite index subgroups are profinitely isomorphic so Theorem~\ref{thm:profinite-isomorphism-adelic} shows that the two \(\mathbf{G}\)-arithmetic pairs \((\Q, \mathbf{H})\) and \((K, \mathbf{H}')\) are locally isomorphic.  We have to show that they are actually isomorphic.  Since the number field \(\Q\) is uniquely determined by its finite adele ring, we have \(K = \Q\).  Moreover, the automorphism group of \(\mathbb{A}_\Q^f\) is trivial, so we have an isomorphism
  \begin{equation} \label{eq:local-equivalence} \mathbf{H} \times_\Q \mathbb{A}^f_\Q \xrightarrow{\ \cong \ } \mathbf{H}' \times_\Q \mathbb{A}^f_\Q \end{equation}
  over \(\mathbb{A}_\Q^f\) from which we need to construct a \(\Q\)-isomorphism \(\mathbf{H} \xrightarrow{\cong} \mathbf{H}'\).  Condition~\eqref{item:dense} in the definition of arithmetic groups shows that \(\mathbb{K} = \R\).  Since \(\mathbf{H}\) and \(\mathbf{H}'\) are simply connected and \(\mathbb{R}\)-isogenous to \(\mathbf{G}\), they are \(\mathbb{R}\)-isomorphic, in particular \(\C\)-isomorphic, and therefore they are \(\Q\)-forms of one another.  The \(\Q\)-forms of \(\mathbf{H}\) are classified up to \(\Q\)-isomorphism by the first noncommutative Galois cohomology set \(H^1(\Q, \Aut_{\overline{\Q}}(\mathbf{H}))\).  Since \(\mathbf{H}\) is neither of type \(A_n\), nor \(D_n\), nor \(E_6\), the outer automorphism group of \(\mathbf{H}\) is trivial, hence \(\Aut_{\overline{\Q}} (\mathbf{H})\) can be identified with the adjoint group \(\overline{\mathbf{H}}\) of \(\mathbf{H}\).  For any \(\overline{\Q}\)-isomorphism \(f \colon \mathbf{H} \xrightarrow{\cong} \mathbf{H}'\), let \(a_\sigma = f^{-1} f^\sigma\) be the \(1\)-cocycle on \(\mathrm{Gal}(\overline{\Q}/\Q)\) with values in \(\Aut_{\overline{\Q}}(\mathbf{H}) = \overline{\mathbf{H}}\).  The corresponding class \([a] \in H^1(\Q, \overline{\mathbf{H}})\) is independent of the choice of \(f\).  The local equivalence in \eqref{eq:local-equivalence} shows that \(\mathbf{H}\) and \(\mathbf{H}'\) are isomorphic at all finite places of \(\Q\).  But as we discussed above, \(\mathbf{H}\) and \(\mathbf{H}'\) are also isomorphic at the infinite place of \(\Q\).  Therefore \([a]\) maps to the distinguished trivial class in \(H^1(K_v, \overline{\mathbf{H}})\) for all places \(v\) of \(\Q\).  Since the \emph{Hasse principle} applies to adjoint groups~\cite{Platonov-Rapinchuk:algebraic-groups}*{Theorem~6.22, p.\,336}, meaning the diagonal map of pointed sets
  \[ H^1(\Q, \overline{\mathbf{H}}) \longrightarrow \prod_v H^1 (\Q_v, \overline{\mathbf{H}}) \]
  is injective, we conclude that \(\alpha\) is the trivial class, whence \(\mathbf{H}\) and \(\mathbf{H}'\) are \(\Q\)-isomorphic.

  Finally we observe (as in Example~\ref{example:non-cocompact-q}), that every non-cocompact lattice \(\Gamma \subset \mathbf{G}(\R)\) has \(K_\Gamma = \Q\), so the second assertion follows immediately.
\end{proof}

\begin{remark} \label{remark:complex-case}
In contrast, we already saw in \cite{Kammeyer-Kionke:rigidity} that for \(\mathbb{K} = \C\), all lattices in \(\mathbf{G}(\C)\), cocompact or not, are profinitely solitary if the complex group \(\mathbf{G}\) has type \(E_8\), \(F_4\), or \(G_2\).  But if the complex group \(\mathbf{G}\) has one of the other types (except possibly \(E_6\) and \(A_n\) where CSP is open), then a non-cocompact lattice \(\Gamma \subset \mathbf{G}(\C)\) is not necessarily profinitely solitary.  Indeed, \(K_\Gamma\) is an imaginary quadratic field and now inner twists over \(\mathfrak{p}\)-adic fields exist.  Hence \(\mathbf{G}\)-arithmetic pairs which are non-isomorphic but locally equivalent to \((K_\Gamma, \mathbf{H}_\Gamma)\) occur as \(K_\Gamma\)-groups which have the same two non-isomorphic \(\Q_p\)-forms over split primes \(p\) as \(\mathbf{H}_\Gamma\), but in different order; compare the proof of \cite{Kammeyer-Kionke:rigidity}*{Theorem~1.1}.
\end{remark}

\end{document}